\documentclass[11pt,a4paper,reqno]{amsart}
\usepackage{amscd,amsmath,amssymb,amsthm,amsfonts,epsfig,graphics}
\usepackage{booktabs}
\headsep=1.2500cm \numberwithin{equation}{section}
\newtheorem{theorem}{Theorem}[section]
\newtheorem{lemma}[theorem]{Lemma}
\newtheorem{proposition}[theorem]{Proposition}
\newtheorem{corollary}[theorem]{Corollary}
\newtheorem{example}{Example}[section]

\theoremstyle{definition}

\theoremstyle{notation}

\begin{document}

 \title[some classes of graphs and their products as frame graphs]
{Characterization of some classes of graphs and their products as frame graphs}



\author[F. Abdollahi]{F. Abdollahi} \author{H. Najafi }

 \address{Department of Mathematics, College of Sciences, Shiraz University, Shiraz 7187919556, Iran.}

\email{hnajafi@shirazu.ac.ir (H. Najafi) and
abdollahi@shirazu.ac.ir (F. Abdollahi).}





\maketitle



 \begin{abstract}
To each finite frame $\varphi$ in an inner product space $\mathcal{H}$ we associate a simple graph $G(\varphi)$, called {\it  frame graph}, with the vectors of the frame as vertices and there is an edge between vertices $f$ and $g$ provided that $\langle f, g\rangle \neq 0$ \cite{AN}. 
In this paper the relation between the order of $G(\varphi)$ and the dimension of $\mathcal{H}$ is investigated for some well-known classes of graphs and their products.
\end{abstract}

 \section{Introduction}
The study of frames, using the properties of graphs, is an
exciting research topic and hopefully will become mutually
useful for both frame and graph theory as well as in computer science.
For example, in \cite{strohmer, Holmes, Bodman2005}
the relation between equiangular tight frames and
graphs was observed. A one-to-one correspondence
between a subclass of equiangular tight frames and
regular two-graphs was offered in \cite{Holmes}
and another one between real equiangular frames
of $n$ vectors and graphs of order $n$ was given in
\cite{Waldron}. The authors of \cite{Sustik} found
some restrictions on the existence of real equiangular
tight frames by an equivalence between equiangular tight
frames and strongly regular graphs with certain parameters. 

 In \cite{AN} we defined a natural connection between frames and graphs.
This connection is made by the zero-nonzero pattern of the
correlation between different elements of the frame.
More precisely, for a finite frame $\varphi$ in
an inner product space $\mathcal{H}$ we associate
a simple graph $G(\varphi)$, termed the {\it frame graph},
with the elements of a frame as vertices, two distinct
vertices are adjacent if and only if the respective
vectors are non-orthogonal. A simple graph $G$ is said
to be {\it frame graph in space $\mathcal{H}$}
if there exists a frame $\varphi$ for $\mathcal{H}$ such that $G(\varphi)=G$.
In \cite{AN} we studied some basic properties of frame graphs and identified some classes of tight frame graphs (a frame graph is a {\it tight frame graph} if the associated frame is tight). Among other things it is shown that complete graphs and the join of each graph with itself are tight frame graphs, whereas non-trivial trees and cycles with at least than six vertices are not. 

This interesting definition and the corresponding observations have received an extensive interest by many researchers, see \cite{Wal, Chi, ChSh, FG, IM} for example.  
Investigating the relation between the dimension of
$\mathcal{H}$ and the graph-theoretic properties of $G$
is the main purpose of this paper. This problem has a deep connection
with a well known problem in graph theory,
namely \textit{minimum positive semidefinite rank}.

The outline of the paper is as follows.
In Section $2$, after fixing some notation
and definitions, we discuss the relation between
the dimension of $\mathcal{H}$ and the
\textit{minimum positive semidefinite rank problem}
which considers the minimum rank over all positive semidefinite
Hermitian matrices whose $ij$th entry (for $i\neq j$)
is nonzero whenever $\{i, j\}$ is an edge in $G$ and is zero
otherwise. By this relation
some well known classes of graphs such
as trees and their complements,
cycles and their complements, complete
and complete bipartite graphs will be
characterized as frame graphs. In Section
$3$, the relation between $dim(\mathcal{H})$
and the order of a graph will be studied for
the join of two graphs and also for the corona,
Cartesian and strong product of some well-known
classes of graphs. As a result we find the minimum
positive semidefinite rank of some classes of
graphs which, as far as we know, are new results.
The results are summarized in Table \ref{table}.

\begin{table}[h]
\centering{
\begin{tabular}{rcccc}
Result \# & $G$ & Order & $mr_+^{\mathbb{C}}(G)$ & $mr_+^{\mathbb{R}}(G)$ \\ \toprule
\cite{AIM}, \ref{T33} & $P_n \boxtimes P_m$ & $nm$ & $(n-1)(m-1)$ & $(n-1)(m-1)$ \\
\ref{T37} & $T \Box K_n$ & $nm$ & $mn-n$ & $mn-n$ \\
\ref{C38} & $C_3 \Box P_n$ & $3n$ &$ 3n-3$ & $3n-3$ \\
\ref{T310} & $K_n \Box K_m$ &$nm$ & $\leq n+m-1$ & $n+m-2$ or $n+m-1$ \\
\ref{c313} & $T \circ T'$& $mm'+m$ &$mm'-1$ & $mm'-1$\\
\ref{c313} & $T \circ K_n, n \geq 2 $ & $mn+m$ & $2m-1$ & $2m-1$\\
\ref{c313} & $K_n \circ T$ & $mn+n$ & $nm-n+1$ & $nm-n+1$ \\
\cite{AIM}, \ref{c313} & $K_n \circ K_m$ & $mn+n$ &$n+1$& $n+1$ \\
\ref{c313} & $C_n \circ T$ & $mn+n$ &$nm-2$& $nm-2$ \\
\ref{c313} & $T \circ C_n$ & $mn+m$ &$m(n-1)-1$& $m(n-1)-1$ \\
\cite{peters}, \ref{c313} & $C_n \circ K_m$ & $mn+n$ &$2n-2$& $2n-2$\\
\ref{c313} & $K_m \circ C_n$ & $mn+m$ &$m(n-2)+1$& $m(n-2)+1$ \\
\ref{c313} & $C_n \circ C_m$ & $mn+n$ &$n(m-1)-2$& $n(m-1)-2$ \\ \bottomrule 
\end{tabular} \vspace{1mm}
}
\caption{Summary of minimum positive semidefinite rank results established in this paper ($T$ and $T'$ are trees with $|T|=m$ and $|T'|=m'$).} \label{table} 
\end{table}

\section{Definitions, Preliminaries and basic facts}\

The following definitions and facts are
standard and can be found in any text book about
graph theory, see for example \cite{Diestel,Bolobas, West}.

Recall that a {\it graph} $G=(V, E)$ consists
of a set of vertices $V$ and a
set of edges $E$, where the elements of $E$ are
two-element sets of vertices.
The {\it order} of a graph $G$, denoted by $|G|$, is
the number of vertices. Distinct vertices $f_i$ and $f_j$
is said to be {\it adjacent} if $\{f_i, f_j\}$ is an edge.
The {\it degree} of each vertex is the number of vertices adjacent to it and the minimum degree among all vertices is denoted by
$\delta(G)$. Note that
our graphs are simple, i.e., multiple edges
and loops do not appear in them and all the edges are undirected.
A {\it path} is a simple graph whose vertices can be ordered so
that two vertices are adjacent if and only if they are consecutive
in the list. A {\it cycle} is a graph with an equal number of
vertices and edges whose vertices can be placed around a circle so
that two vertices are adjacent if and only if they appear
consecutively along the circle. A cycle of order $n$ is denoted by $C_n$.

A graph $G$ is {\it connected} if each pair of vertices in $G$
belongs to a path in $G$. A {\it connected component} of $G$
is a maximal connected subgraph of $G$. A connected
graph without a cycle is called {\it a tree}.\\
An {\it independent} set in a graph is a set of pairwise
nonadjacent vertices. The {\it independence} number of a graph
$G$, $\alpha(G)$, is the maximal size of an independent set of
vertices. A graph $G$ is {\it complete} or a {\it clique} if $\alpha(G)=1$.
The complete graph of order $n$ is denoted by $K_n$.
The graph $G=(V, E)$ is called {\it bipartite} if $V$ is the union
of two disjoint (possibly empty) independent sets called partite
sets of $G$. If every vertex of the first set with $n$
elements is adjacent to every vertex of the second set
with $m$ elements, then it is called a {\it complete bipartite} graph
and denoted by $K_{n,m}$ \\ 
By the {\it complement} $\overline{G}$ of a graph
$G$, we mean the graph on the same vertex set where
two vertices are adjacent if and only if they are not
adjacent in $G$.\\
Given a graph $G=(V, E)$, the graph $H=(V', E')$ with
$V' \subseteq V$ and $E' \subseteq E$ is called
a {\it subgraph} of $G$. An {\it induced} subgraph $H$ is
a subgraph in which two vertices are adjacent if and
only if they are adjacent in $G$. 

 Suppose that $G$ is labeled and that $G_1, . . . , G_k$
are (labeled) subgraphs of $G$, that is, each $G_i , i = 1, . . . , k$
is the result of deleting some edges and/or vertices from $G$.
We say that $G_1, . . . , G_k$ cover $G$ if each edge (vertex) of
G is an edge (vertex) of at least one $G_i, 1 \leq i \leq k$ or
$G=\bigcup_{i=1}^{k} G_i$.
The cover $G_1, . . . , G_k$ of $G$ is called a clique cover of $G$
if each of $G_1, . . . , G_k$ is a clique of $G$.
The {\it clique cover number} of G, $cc(G)$, is
the minimum value of $k$ for
which there is a clique cover $G_1, . . . , G_k$ of $G$.

 A graph is said to be {\it chordal} if it has no induced cycles $C_n$
with $n \geq 4$. As examples trees and complete graphs are chordal.

 For an $n\times n$ Hermitian matrix $A$, the graph of $A$, denoted
by $\mathcal{G}(A)$, is the graph with vertices $\{1, 2, \ldots, n\}$
such that two distinct vertices $i$ and $j$ are adjacent if and only if
the $i j$-th entry of $A$ is non-zero.
The set of all $n\times n$ complex Hermitian positive semidefinite matrices will be
denoted by $\mathcal{H}_{n}^{+}$ and its subset consisting of all real matrices
denoted by $S_{n}^{+}$.

 The {\it real minimum positive semidefinite rank} of $G$, ${\it mr}_{+}^{\mathbb{R}}(G)$,
and {\it complex Hermitian minimum positive semidefinite rank}, ${\it mr}_{+}^{\mathbb{C}}(G)$ are defined by
$$\textrm{mr}_{+}^{\mathbb{R}}(G)=\min\{rank(B): B \in S_{n}^{+} \ and \ \mathcal{G}(B)=G \},$$ 
and
$$\textrm{mr}_+^{\mathbb{C}}(G)=\min\{rank(B): B \in \mathcal{H}_{n}^{+} \ and \ \mathcal{G}(B)=G \}.$$
If $mr_{+}^{\mathbb{R}}(G)=mr_{+}^{\mathbb{C}}(G)$ , then
we denote the common value $mr_{+}^{\mathbb{R}}(G)=mr_{+}^{\mathbb{C}}(G)$
by $mr_+(G)$.

 Clearly ${\it mr}_{+}^{\mathbb{C}}(G) \leq {\it mr}_{+}^{\mathbb{R}}(G)$.
A graph on $16$ vertices for which these parameters are
not identical was presented in \cite{Barioli}. For a graph $G$ of order $n$,
it is well known and straightforward that ${\it mr}_{+}^{\mathbb{R}}(G) \leq n-1$.
Other fundamental facts, that will be of interest to us, are that
${\it mr}_{+}(G)= n-1$ if and only if $G$ is a tree
and ${\it mr}_{+}(C_n)=n-2$\cite{Holst}.
For more details on minimum positive semidefinite rank
we refer the reader to \cite{Holst, Matthew, Fallat, AIM, Hackney, Barioli}.

 A {\it finite frame} for a finite dimensional
Hilbert space $\mathcal{H}$ (or inner product space)
is a finite sequence $\{f_i\}_{i=1}^{n}$ in $\mathcal{H}$ such that
there exist constants $0<A\leq B< \infty $ with the property that
$$A\parallel f\parallel^ 2 \leq \sum_{i=1}^{n}|\langle f,f_i\rangle|^2\leq
B\parallel f\parallel^2$$
holds for all $f \in \mathcal{H}$ \cite{Chris}.

Given a frame $\varphi=\{f_i\}_{i=1}^{m}$ in $\mathbb{R}^{n}$ or $\mathbb{C}^{n}$.
Let $B$
be the $n \times m$ matrix whose $j$th column is $f_j$. Then $B^{T}B$ ( $B^{*}B$)
is a positive semidefinite martix
called the {\it Gramian matrix} of $\varphi$ and
denoted by $Gram(\varphi)$($B^{*}$
is the conjugate transpose of B). The rank of $B$,
$n$, is equal to the rank of $Gram(\varphi)$ \cite[Theorem 7.2.10]{Horn}.

 Let $\varphi$ be a finite frame for the inner product space $\mathcal{H}$.
We associate a simple graph $G({\varphi})$, called a {\it frame graph},
whose vertices are the elements of $\varphi$ and two distinct
vertices $a$ and $b$ are adjacent if and only if $\langle
a,b\rangle \neq0 $. A simple graph $G$ is called
a {\it frame graph} in the space $\mathcal{H}$
if there exists a frame $\varphi$
for $\mathcal{H}$ such that $G(\varphi)=G$.
In this paper all graphs are non-trivial and connected, and so the associated frames do not include zero vectors.

For a given graph $G$, we interested to find the inner product $\mathcal{H}$ so that
there exists a frame $\varphi$ for $\mathcal{H}$ such that
$G({\varphi})= G$. At first, we show that this problem is
deeply connected with the well known \textit{minimum positive semidefinite rank problem}.
We will find the minimum positive semidefinite rank for join of graphs and Cartesian, corona and strong products of some
well known classes of graphs and use them to characterize these products as frame
graphs.

\begin{lemma} \label{L21}
For a graph $G$, $mr_{+}^{\mathbb{R}}(G)=k$ (respectively, $mr_{+}^{\mathbb{C}}(G)=k$)
if and only if $k$ is the minimum number such that $G$ is a frame graph in a real (respectively, complex)
inner product space of dimension $k$. 
\end{lemma}
\begin{proof}
Let $A$ be a positive semidefinite real (complex) $n \times n$ matrix of rank $k$ such that
$\mathcal{G}(A)=G$. The matrix $A$ has a decomposition of the form
$A= B^{T}B$. The
columns of $B$ constitute a frame $\varphi$ in a real (complex)$k$-dimensional
space such that $G(\varphi)=G$. On the other hand,
the Gramian matrix of each frame $\varphi$ in a real (complex)
space $\mathcal{H}$ describes $G(\varphi)$, or $\mathcal{G}(Gram(\varphi))=G(\varphi)$,
and its rank is $dim(\mathcal{H})$. Hence the lowest
dimension of a real (complex) space in which a frame associated to the
graph $G$ is exactly the real (complex Hermitian) minimum positive semidefinite
rank.
\end{proof}

\begin{lemma} \label{L22}
Let $G$ be a frame graph in $\mathbb{R}^n$ (respectively, $\mathbb{C}^n )$ with $m>n$ vertices.
Then $G$ is a frame graph in $\mathbb{R}^{n+1}$ (respectively, $\mathbb{C}^{n+1} )$.
\end{lemma}
\begin{proof}
Let $\varphi=\{f_1, f_2,\ldots,f_m\}$ be a frame
for $\mathbb{R}^n$ (respectively, $\mathbb{C}^n )$ such that $G(\varphi)=G$. Since $\varphi$ is a
frame, a subset of $\varphi$ with $n$ elements spans $\mathbb{R}^n$ (respectively, $\mathbb{C}^n$).
Then, without loss of generality, we may assume that $\{f_1, f_2,\ldots,f_n\}$
spans $\mathbb{R}^n$ (respectively, $\mathbb{C}^n$). Clearly
$$\varphi':= \{(f_1,0), (f_2,0),\ldots, (f_n,0), (f_{n+1},1), (f_{n+2},0),\ldots, (f_m,0)\},$$
spans $\mathbb{R}^{n+1}$ (respectively, $\mathbb{C}^{n+1} )$
and $G(\varphi)=G(\varphi')$.
Hence $G$ is a frame graph in $\mathbb{R}^{n+1}$ (respectively, $\mathbb{C}^{n+1} )$.
\end{proof}

 For a given graph $G$ on $n$ vertices,
the above lemmas conclude that
$mr_{+}^{\mathbb{R}}(G)=k$ (respectively, $mr_{+}^{\mathbb{C}}(G)=k$)
if and only if $G$ is just a frame graph in the real (respectively, complex) inner product spaces of dimension
$k, k+1,...,n$. For example, $G$ is just a frame graph in
spaces of dimension $n-1$ and $n$ if and only if $G$ is a tree. Another example
is a cycle of order $n$, $C_n$, which is just a frame graph in the spaces of dimension
$n-2, n-1$ and $n$.
Trees and cycles, as two well known classes of graphs,
are not frame graph in the spaces of lower dimensions, whereas complete
graphs are frame graphs in all possible spaces because $mr_{+}(K_n)=1$.
Clearly complete graphs are the only connected
graphs which are frame graphs in all possible spaces.
Let $G$ be a non-complete connected graph and
$\varphi$ be a frame for an inner product
space $\mathcal{H}$ such that $G(\varphi)=G$.
Then the dimension of $\mathcal{H}$ is at least two.
The following example introduces two
classes of connected graphs which achieve this bound.
\begin{example}\label{e23}
Let $G$ be a simple graph obtained from $K_n$ $(n \geq 3)$
by deleting an edge and let
$\{e_1, e_2\}$ be the standard orthogonal basis of $\mathbb{C}^{2}$
($\mathbb{R}^{2}$). Then
the frame $\varphi=\{e_1+e_2, e_1-e_2\} \bigcup \{e_1\}_{i=1}^{n-2}$
makes $G$ a frame graph in $\mathbb{C}^{2}$ ($\mathbb{R}^{2}$), i.e,
$mr_{+}(G)=2$.
Another graph is $H_n$, the ($n-2$)-regular graph on $n$ vertices. It is known
that $mr_{+}(H_n)=2$ (see Proposition
5.2 of \cite{Matthew}). Both $G$ and $H_n$ are frame
graphs in the inner product spaces
of dimension $2, 3,..., n-1$ and $n$. 
\end{example}

 Despite trees and complete graphs which are frame graphs in a minimum
and maximum number of spaces, respectively, complete bipartite graphs
with the parts of the same size are frame graphs in a medium number of possible spaces.
Indeed, since the independence number of each complete bipartite graph $K_{m,n} (m\geq n)$
is $m$ and it has no isolated vertex, it cannot be a frame graph in an inner product
space of dimension less than $m$. Let $\{e_i\}_{i=1}^{m}$ be
the standard orthogonal basis of $\mathbb{R}^{m}$.
The set $\{e_{i}\}_{i=1}^{m} \bigcup \{e_{1}+e_{2}+\ldots+
e_{i-1}+\frac{2-n}{2}e_{i}+e_{i+1}+\ldots+e_{m}\}_{i=1}^{n}$ $(n \neq 2)$ is a
frame for $\mathbb{R}^m$ and its frame graph is $K_{n,m}$.
Now Lemma \ref{L22} guaranties that $K_{m,n}$ is just
a frame graph in the spaces of dimension $m,
m+1,..., n+m-1, n+m$.
Proposition 2.2 of \cite{Hackney} shows that this result holds for a larger class
of bipartite graphs. We include it here:

 \begin{proposition}\label{p24}
Let $G$ be a bipartite graph which contains $K_{n,n}$
as an induced subgraph and its partite sets $U$ and $V$
are of size $m$ and $n$ where $m \geq n$. Then $G$ is just a frame graph in the inner product spaces of dimension $m, m+1,...,m+n-1, m+n$.
\end{proposition}

 \section{Product of graphs and frame graph}

 As we saw in the previous section, to characterize a
graph as frame graph we just need
to find its minimum positive semidefinite rank.
Some well known classes of graphs such as trees
and their complements, cycles and their
complements, complete and complete bipartite graphs have known
minimum positive semidefinite rank
(see \cite{Li}), so the dimension of all the spaces which
they are frame graphs in them can be obtained by Lemma \ref{L22}.
In this section, we try to find the minimum positive semidefinite rank
of the product of some prominent classes of graphs and use them to
characterize these graphs as frame graphs. The following definition and proposition, which can
be found in \cite{Hackney},
provides a lower bound for $ mr_+(G)$ and will be useful
later in this work.

Let $G=(V, E)$ be a connected graph and $S = \{v_1, . . . , v_m\}$
be an ordered set of vertices of $G$.
Denote by $G_k$ the subgraph induced by $v_1, v_2, . . . , v_k$
for all $k= 1, 2,..., m $. Let $H_k=(V_k, E_k)$
be the connected component of $G_k$ such that $v_k \in V_k$.
If for each $k$, there exists $w_k \in V$ such that
$w_k \neq v_l$ for $l \leq k$, $\{v_k, w_k\} \in E$ and
$\{w_k, v_l\} \notin E$ for all $v_l \in V_k$ with $l \neq k$,
then $S$ is called a vertex set of ordered subgraphs (or $OS$-vertex set).
The {\it OS-number} of a graph $G$, denoted by $OS(G)$, is the maximum
cardinality among all $OS$-vertex sets of $G$.
For example $OS(T)=m-1$ where $T$ is a tree on $m$ vertices
(see \cite[Proposition 3.9]{Hackney}) and, clearly, $OS(K_n)=1$.
Note that the
$OS$-number is not defined for a single vertex, $K_1$, and so, in this
section, we assume that all graphs have more than one vertex.

 \begin{proposition}\label{p31} \cite[Proposition 3.3 ]{Hackney} and
\cite[Observation 3.14.]{Fallat}.
Let $G$ be a connected graph.
Then $OS(G) \leq mr_+^{\mathbb{C}}(G) \leq mr_+^{\mathbb{R}}(G) \leq cc(G)$.
\end{proposition}

The {\it vertex connectivity} $\kappa(G)$ of a
connected graph $G=(V, E)$ is the minimum size of
$S \subseteq V$ such that $G-S$ is disconnected or
a single vertex. The following result provides an upper bound
for the real minimum positive semidefinite rank.

\begin{proposition}\label{p32} \cite{Lovasz1989, Lovasz2000}.
For each graph $G$ on $n$ vertices, $mr_{+}^{\mathbb{R}}(G) \leq n-\kappa(G).$
\end{proposition}

 \subsection{Strong product}
The {\it strong product} of two graphs $G_1=(V_1, E_1)$ and
$G_2=(V_2, E_2 )$, denoted by $G_1 \boxtimes G_2$, is the graph with the vertex set
$V_1\times V_2$ such that $(u, v)$ is adjacent to $(u', v')$
if and only if (1) $u=u'$ and $\{v, v'\} \in E_2 $
or (2) $v=v'$ and $\{u, u'\} \in E_1$ or (3)
$\{u, u'\} \in E_1$ and
$\{v, v'\} \in E_2$. 

 \begin{theorem}\label{T33} Let $G$ be the strong product of
$P_n$ and $P_m$, i.e., $G=P_n \boxtimes P_n$. Then $G$
is just a frame graph in the inner product spaces of
dimension $(n-1)(m-1), (n-1)(m-1)+1,..., mn-1, mn$.
\end{theorem}
\begin{proof} If we prove $mr_+(G)=(n-1)(m-1)$, then
the ''lifting'' argument of Lemma \ref{L22} completes the proof.
$G$ has a clique cover of $(n-1)(m-1)$ copies of $K_4$,
so, by Proposition \ref{p31}, $mr_+^{\mathbb{R}}(G) \leq (n-1)(m-1)$. Let $\{s_1, s_2,..., s_n\}$
and $\{t_1, t_2,..., t_m\}$ be the vertex sets of $P_n$ and $P_m$, respectively.
The ordered set $\bigcup_{i=1}^{n-1} \bigcup_{j=1}^{m-1}\{v_{i,j}\}$, where
$v_{i,j}=(s_i, t_j)$, with $w_{i,j}=(s_{i+1}, t_{j+1})$
is an $OS$-vertex set of length $(n-1)(m-1)$. Therefore $(n-1)(m-1) \leq OS(G)$.
Now Proposition \ref{p31} implies $(n-1)(m-1) \leq mr_+^{\mathbb{C}}(G)$.
\end{proof}
\subsection{Joins of graphs}
The {\it join} $G_1 \vee G_2$ of the graphs $G_1=(V_1, E_1)$
and $G_2=(V_2, E_2)$ with disjoint vertex and
edge sets is the graph with vertex
set $V_1 \bigcup V_2$ and edge set $E_1 \bigcup E_2$ together with
all the edges joining $V_1$ and $V_2$.

Suppose $G$ is decomposable into two graphs, $G_1=(V_1, E_1)$
and $G_2=(V_2, E_2)$,
sharing only one vertex $v$ such that if $u \in V_1$
and $w \in V_2$, then $\{u, w\} \in E$ only if $u=v$ or $w=v$.
Then $G_1$ and $G_2$ are joined at vertex $v$ and
$G$ is denoted by $G= G_1 . G_2$.

Let $G$ and $H$ be two connected graphs on two or more vertices.
Proposition $2.4$ of \cite{Hackney} states that $mr_+^{\mathbb{C}}(G\vee H)=max\{mr_+^{\mathbb{C}}(G), mr_+^{\mathbb{C}}(H)\}$
and by Theorem $3.4$ of \cite{Matthew} we have $mr_+^{\mathbb{C}}(G.H) = mr_+^{\mathbb{C}}(G)+mr_+^{\mathbb{C}}(H)$.
These, combined with Lemma \ref{L22}, give the following theorem.

\begin{theorem} \label{T34}
Let the graph $G$ be just a frame graph in
spaces of dimension $n', n'+1,.., n-1, n$
and $H$ be just a frame graph in the
spaces of dimension $m', m'+1,.., m-1, m$ where $n' \geq m'$. The followings are hold:
\begin{enumerate}
\item[(i)] $G\vee H$ is just a frame graph in the spaces of dimension $n', n'+1,...,m+n-1$ and $m+n$.
\item[(ii)] $G.H$ is just a frame graph in the spaces of dimension $n'+m',n'+m'+1,...,m+n-2$ and $m+n-1$.
\end{enumerate}
\end{theorem}

\subsection{Cartesian product}
The {\it Cartesian product} of two graphs $G_1=(V_1, E_1)$ and
$G_2=(V_2, E_2 )$, denoted by $G_1 \Box G_2$, is the graph with vertex set
$V_1\times V_2$ such that $(u, v)$ is adjacent to $(u', v')$
if and only if (1)$u=u'$ and $\{v, v'\} \in E_2 $
or (2) $v=v'$ and $\{u, u'\} \in E_1$ . 
The set of vertices associated with (the same) $OS$-vertex
set in each copy of $G$ or $H$ is an $OS$-vertex set for $G\Box H$. Then we have the following lemma.
\begin{lemma} \label{l35}
Let $G$ and $H$ be the graphs of order $n$ and $m$, respectively. Then
$$\max \{nOS(H), mOS(G)\} \leq OS(G \Box H).$$
\end{lemma}

 The connectivity of the graph $G\Box H$ can be determined by
knowing the minimum degree and the connectivity of the components
$G$ and $H$.
\begin{proposition} \label{p36} \cite{Spacapan}
Let $G$ and $H$ be the graphs on $n$ and $m$
vertices, respectively. Then $\kappa(G \Box H)= \min
\{n \kappa(H), m \kappa(G), \delta(G)+\delta(H)\}.$
\end{proposition}

 \begin{theorem} \label{T37} The graph $G= T \square K_m$,
where $T$ is a tree on $n$ vertices,
is just a frame graph in the inner product spaces of
dimension $mn-m, mn-m+1,..., mn-1, mn$.
\end{theorem}
\begin{proof} It is sufficient to show that $mr_+(G)=mn-m$
and use the "lifting argument'' of Lemma \ref{L22}.

 By Lema \ref{l35}, $mn-m \leq OS(G)$, so Proposition \ref{p31}
implies $mn-m \leq mr_+^{\mathbb{C}}(G)$. On the other hand,
using Propositions \ref{p32} and \ref{p36},
we have $mr_{+}^{\mathbb{R}}(G) \leq mn-m$. Using
$mr_{+}(G) \leq mr_{+}^{\mathbb{R}}(G)$ implies that
$mr_{+}^{\mathbb{C}}(G) = mr_{+}^{\mathbb{R}}(G)=mn-m$ which completes
the proof.
\end{proof}
It is a conjecture that $mr_{+}(C_m \Box P_n)= mn - \min \{m, 2n\}$
(See \cite{peter}). This conjecture is valid for the case $n=2$(\cite{peter}).
The above theorem shows that it also holds
for the case $m=3$.
\begin{corollary} \label{C38}
$mr_{+}(C_3 \Box P_n)=3n-3$.
\end{corollary}
At the end of this subsection we consider the graph $K_n \Box K_m$. Unfortunately
the minimum positive semidefinite rank of this graph is not known until now, and hence
we cannot characterize it as frame graph. Another conjecture in \cite{peter}
states that $mr_{+}^{\mathbb{R}}(K_n \Box K_m)= n+m-2$. In what follows
we show that $mr_{+}^{\mathbb{R}}(K_n \Box K_m)$ is one of the numbers $n+m-2$ or $n+m-1$,
therefore $K_n \Box K_m$ is a frame graph in the real spaces of dimension $n+m-1, m+n,..., nm$
and the dimension $n+m-2$ remains unsolved.

If $A$ is an $n \times n$ matrix and $B$ is an $m\times m$
matrix, then $A \otimes B$, the Kronecker product, is the $n \times n$
block matrix whose $(i, j)$th block is the $m \times m$
matrix $a_{ij}B$.
Let $G$ and $H$ be the graphs on $n$ and $m$ vertices, respectively, and
let $A$ and $B$ be two matrices such that $\mathcal{G}(A)=G$ and $\mathcal{G}(B)=H$. Then $\mathcal{G}(A \otimes I_m + I_n \otimes B )= G \Box H$ (see \cite{Godsil}).
The matrix $A \otimes I_m + I_n \otimes B$ known as {\it Kronecker sum} and denoted by $A \oplus B$. If $\lambda$ is an eigenvalue of $A$ and $\mu$ is an eigenvalue
of $B$, then $\lambda + \mu$ is an eigenvalue of $A \oplus B$
and any eigenvalue of $A \oplus B$ arises as such a sum
of eigenvalues of $A$ and $B$ \cite{Godsil}.
We use these facts to proof the next lemma. 

 \begin{lemma} \label{L39}
For the graphs $G$ and $H$ on $n$
and $m$ vertices, respectively,
$$mr_{+}^{\mathbb{R}}(G \Box H) \leq
m mr_{+}^{\mathbb{R}}(G)+ n mr_{+}^{\mathbb{R}}(H)- mr_{+}^{\mathbb{R}}(G) mr_{+}^{\mathbb{R}}(H).$$
\end{lemma}
\begin{proof} Let $A$ and $B$ be two matrices
in $ S_{n}^{+}$ and $S_{m}^{+}$, respectively, such that $\mathcal{G}(A)=G$, $\mathcal{G}(B)=H$, $rank(A)= mr_{+}^{\mathbb{R}}(G)$, and $rank(B)=mr_{+}^{\mathbb{R}}(H)$.
Then $A$ has $n-mr_{+}^{\mathbb{R}}(G)$ zero eigenvalues,
$B$ has $m-mr_{+}^{\mathbb{R}}(H)$ zero eigenvalues and all the other eigenvalues
of $A$ and $B$ are positive and non zero. Since the eigenvalues of
$A \oplus B$ are the sum of the eigenvalues of $A$ and $B$, $A \oplus B$
has exactly $ m mr_{+}^{\mathbb{R}}(G)+ n mr_{+}^{\mathbb{R}}(H)- mr_{+}^{\mathbb{R}}(G)
mr_{+}^{\mathbb{R}}(H)$ positive non zero eigenvalues, and so $rank(A \oplus B)=
m mr_{+}^{\mathbb{R}}(G)+ n mr_{+}^{\mathbb{R}}(H)- mr_{+}^{\mathbb{R}}(G) mr_{+}^{\mathbb{R}}(H)$.
The matrix $A \oplus B$ belongs to $S^{+}_{nm}$ and describes $G \Box H$,
i.e, $\mathcal{G}(A \oplus B)= G\Box H$, so
$mr_{+}^{\mathbb{R}}(G \Box H) \leq rank(A \oplus B)=
m mr_{+}^{\mathbb{R}}(G)+ n mr_{+}^{\mathbb{R}}(H)- mr_{+}^{\mathbb{R}}(G) mr_{+}^{\mathbb{R}}(H).$
\end{proof}

 The minimum rank of a graph $G$ (over $\mathbb{R}$) is defined to be
$$\textrm{mr}(G)= \min\{rank(A): A \in S_{n} \ and \ \mathcal{G}(A)=G \}, $$
where $S_n$ is the set of symmetric matrices over $\mathbb{R}$.
Clearly $$ mr(G)\leq mr_+^{\mathbb{R}}(G).$$
\begin{theorem} \label{T310}
$mr_{+}^{\mathbb{R}}(K_n \Box K_m)$ is either $n+m-2$
or $n+m-1$.
\end{theorem}
\begin{proof}
It is known that $mr(K_n \Box K_m)=m+n-2$ (see \cite[Corollary 3.11]{AIM}).
Since $ mr(K_n \Box K_m)\leq mr_+^{\mathbb{R}}(K_n \Box K_m)$,
$n+m-2 \leq mr_+^{\mathbb{R}}(G)$. On the other hand, by Lemma
\ref{L39}, $mr_{+}^{\mathbb{R}}(K_n \Box K_m)$ cannot be more than $n+m-1$.
\end{proof}

 \subsection{Corona product}
The {\it corona product} of $G_1=(V_1, E_1)$ with
$G_2=(V_2, E_2)$, denoted by $G_1 \circ G_2$,
is the graph of order $|V_1||V_2|+|V_1|$ obtained by
taking one copy of $G_1$ and $|V_1|$ copies of $G_2$, and joining
all the vertices in the $i$th copy of $G_2$ to the $i$th vertex of
$G_1$ \cite{Harary}.\\
The following proposition helps us to provide a short proof
for the next theorem.
\begin{proposition} \label{p311}\cite[Observation 3.14 ]{Fallat} \label{o42}
Let $G$ be a graph and $G_1, G_2,...,G_n$ be a cover
of (labeled)subgraphs of it, i.e., $G= \bigcup _{i=1}^{n} G_i$.
Then $mr_+^{\mathbb{R}}(G) \leq \sum _{i=1}^{n} mr_+^{\mathbb{R}}(G_i)$.
\end{proposition}
\begin{theorem} \label{T312} Let $G$ be a connected graph
of order $n$ with $OS(G)=mr_+(G)$ and $H$ be a connected
graph with $OS(H)=mr_+(H)$.
Then $mr_+(G \circ H)=n(mr_+(H))+ mr_+(G)$.
\end{theorem}
\begin{proof} Let $\{v_1, v_2,..., v_n\}$
be the vertex set of $G$. $G \circ H= (\bigcup_{i=1}^{n} H\vee v_i) \bigcup G $.
Therefore, by Proposition \ref{o42},
$$mr_+^{\mathbb{R}}(G \circ H) \leq (\sum _{i=1}^{n} mr_+^{\mathbb{R}}(H \vee v_i))+mr_+^{\mathbb{R}}(G)
=n(mr_+(H))+mr_+(G).$$
 On the other hand, the union of $n$ distinct $OS$-vertex sets of the copies of
$H$ along with the $OS$-vertex set of a copy
of $G$ in $G \circ H$ makes an $OS$-vertex set
of length $n OS(H)+ OS(G)$ for $G \circ H$, and so
$n OS(H)+ OS(G) \leq OS(G \circ H)$. Again by using Proposition \ref{p31}, we have $n (mr_+(H))+mr_+(G) \leq mr_+^{\mathbb{C}}(G \circ H)$.
\end{proof}

 It is easy to check that $OS(C_t)=mr_+(C_t)=t-2$ and $OS(K_n)=mr_+(K_n)=1$ and
it is known that if $G$ is a connected chordal graph,
then $OS(G)=mr_+^{\mathbb{C}}(G) = cc(G)$ (see \cite[Corollary 3.8]{Hackney}).
If $T$ is a tree of order $m$, then it is chordal and $OS(T)=mr_+(T)=cc(T)=m-1$. Now the following corollary is a direct consequence of Theorem \ref{T312}.

\begin{corollary} \label{c313}
Let $T$ and $T'$ be trees of order $m \geq 2$
and $m' \geq 2$, respectively. The followings hold.
\begin{enumerate}
\item $mr_+(T \circ T')=mm'-1$. 
\item $mr_+(T \circ K_n)=2m-1 (n \geq 2)$. 
\item $mr_+(K_n \circ T )=nm-n+1$.
\item $mr_+(K_n \circ K_m)=n+1$. 
\item $mr_+(C_n \circ T)=nm-2$. 
\item $mr_+(T \circ C_n)=m(n-1)-1 $. 
\item $mr_+(C_n \circ K_m )=2n-2 $.
\item $mr_+(K_m \circ C_n)=m(n-2)+1.$ 
\item $mr_+(C_n \circ C_{m})=n(m-1)-2.$ 
 \end{enumerate}
\end{corollary}

As far as we know all of the results
of Corollary \ref{c313}
are new with the exception of $4$ and $7$ which
was established earlier in \cite{AIM} and \cite{peters}.

 Note that $mr_+(H_n)=OS(H_n)=2$. Therefore the minimum positive
semidefinite rank of the corona product of
$H_n$ with trees, cycles and complete graphs can be calculated easily.

 The following theorem is an immediate consequence of
Theorem \ref{T312}, Lemma \ref{L21} and Lemma \ref{L22}.

\begin{theorem} \label{T314}
 Let $G$ and $H$ be connected graphs of order
$n$ and $m$, respectively, which are either chordal or cycle graph.
If $G$ is just a frame graph in the spaces of dimension
$n', n'+1,..., n$ and $H$ is just a frame graph in the spaces of
dimension $m', m'+1,..., m$, then $G \circ H$
is just a frame graph in the spaces of dimension
$nm'+n', nm'+n'+1,...,nm+n-1, nm+n$.
\end{theorem}


\end{document}